\documentclass[11pt,a4paper,reqno]{amsart}
\usepackage[latin1]{inputenc}
\usepackage[english]{babel}
\usepackage{amsmath}
\usepackage{amsfonts}
\usepackage{amssymb}
\usepackage{mathrsfs}
\usepackage{latexsym}
\usepackage{yfonts}
\usepackage{natbib}

\usepackage[margin=3cm]{geometry}

\usepackage{color}

\usepackage{graphicx,color}

\usepackage[plainpages=false,colorlinks,hyperindex,bookmarksopen,linkcolor=red,citecolor=blue,urlcolor=blue]{hyperref}

\bibpunct{[}{]}{;}{n}{,}{,}

\newtheorem{theorem}{Theorem}[section]

\theoremstyle{theorem}

\newtheorem{remark}{Remark}[section]

\numberwithin{equation}{section}

\allowdisplaybreaks

\begin{document}

	\title [Spectral functions] {Spectral functions related to some fractional stochastic differential equations}
	
	\author{Mirko D'Ovidio}
	\address{Department of Basic and Applied Sciences for Engineering\newline Sapienza University of Rome\newline via A. Scarpa 10, Rome, Italy}
	\email[Corresponding author]{mirko.dovidio@uniroma1.it}
	
	\author{Enzo Orsingher}
	\address{Department of Statistical Sciences\newline Sapienza University of Rome\newline p.le A. Moro 5, Rome, Italy}
	\email{enzo.orsingher@uniroma1.it}
	
	\author{Ludmila Sakhno}
	\address{Department of Probability Theory, Statistics and Actuarial Mathematics\newline Taras Shevchenko National University of Kyiv, Volodymyrska 64, Kyiv 01601, Ukraine}
	\email{lms@univ.kiev.ua}

	\date{\today}


\begin{abstract}
In this paper we consider fractional higher-order stochastic differential equations of the form
\begin{align*}
\left( \mu + c_\alpha \frac{d^\alpha}{d(-t)^\alpha} \right)^\beta X(t) = \mathcal{E}(t) , \quad t\geq 0,\; \mu>0,\; \beta>0,\; \alpha \in (0,1) \cup \mathbb{N}
\end{align*}		 
where $\mathcal{E}(t)$ is a Gaussian white noise. We derive stochastic processes satisfying the above equations of which we obtain explicitly the covariance functions and the spectral functions.
\end{abstract}
	
	\maketitle

\section{Introduction}
In this paper we consider fractional stochastic ordinary differential equations of different form where the stochastic component is represented by a Gaussian white noise. Most of the fractional equations considered here are related to the higher-order heat equations and thus are connected with pseudoprecesses. 

The first part of the paper considers the following stochastic differential equation
\begin{equation}
	\left(\mu + \frac{d^\alpha}{d (-t)^\alpha} \right)^\beta X(t) \, = \, \mathcal{E}(t), \qquad \beta >0, \, 0<\alpha < 1,\; \mu>0
	\label{sol1-intro-pde}
\end{equation}
where $\frac{d^\alpha}{d (-t)^\alpha}$ represents the upper-Weyl fractional derivative. We obtain a representation of the solution to \eqref{sol1-intro-pde} in the form
\begin{equation}
X(t) \, = \, \frac{1}{\Gamma (\beta)} \int_0^\infty dz \int_0^\infty ds \, s^{\beta -1} e^{-s\mu} \, h_\alpha (z, s)\, \mathcal{E}(t+z) \label{sol1-intro}
\end{equation}
where $h_\alpha(z,s)$, $z,s\geq 0$, is the density function of a positively skewed stable process $H_\alpha(t)$, $t\geq 0$ of order $\alpha \in (0,1)$, that is with Laplace transform
\begin{align*}
\int_0^\infty e^{-\xi z} h_\alpha(z,s)dz= e^{-s \xi^\alpha}, \quad \xi \geq 0.
\end{align*}
For \eqref{sol1-intro}, we obtain the spectral function
\begin{equation}
f(\tau) = \frac{\sigma^2}{\left( \mu^2  + 2 |\tau |^\alpha \mu \cos \frac{\pi \alpha}{2} + |\tau |^{2\alpha} \right)^\beta}, \quad \tau \in \mathbb{R}
\label{spectralfFig}
\end{equation}
and the related covariance function.

The second type of stochastic differential equations we consider has the form 
\begin{equation}
\left(\mu + (-1)^{n} \frac{\partial^{2n}}{\partial t^{2n}} \right)^\beta X(t) \, = \, \mathcal{E}(t), \quad \beta >0,\; \mu >0,\; n\geq 1,  \label{sol2-intro-pde}
\end{equation}
where $\mathcal{E}(t)$ is a Gaussian white noise. The representation of the solution to \eqref{sol2-intro-pde} is
\begin{equation}
X(t) \, = \, \frac{1}{\Gamma \left(\beta \right)} \int_0^\infty dw \, w^{\beta -1} e^{-\mu w} \int_{-\infty}^{+\infty} dx\, u_{2n}(x,w) \mathcal{E}(t+x) \label{sol2-intro}
\end{equation}
where $u_{2n}(x,w)$, $x \in \mathbb{R}$, $w\geq 0$ is the fundamental solution to $2n$-th order heat equation
\begin{equation}
\frac{\partial u}{\partial w}(x,w) = (-1)^{n+1} \frac{\partial^{2n} u}{\partial x^{2n}}(x,w)
\label{h-o-pde}
\end{equation}
The autocovariance function of the process \eqref{sol2-intro} can be written as 
\begin{equation}
\mathbb{E} X(t)X(t+h) = \frac{\sigma^2}{\Gamma (2\beta )}\int_0^\infty dw\,w^{2\beta -1} e^{-\mu w} \, u_{2n}(h, w) = \frac{\sigma^2}{\mu^{2\beta}} \mathbb{E} u_{2n}(h, W_{2\beta})  \label{cov-sol2-intro}
\end{equation}
where $W_{2\beta}$ is a gamma r.v. with parameters $\mu$ and $2\beta$. The spectral function $f(\tau)$ associated with \eqref{cov-sol2-intro} has the fine form
\begin{equation}
f(\tau) = \frac{\sigma^2}{(\mu + \tau^{2n})^{2\beta}}  
\end{equation}
For $n=1$, \eqref{h-o-pde} is the classical heat equation, $u_{2}(x,w) = \frac{e^{-\frac{x^2}{4w}}}{\sqrt{4\pi w}}$ and, from \eqref{cov-sol2-intro} we obtain an explicit form of the covariance function in terms of the modified Bessel functions. In connection with the equations of the form \eqref{h-o-pde} the so-called pseudoprocesses, first introduced at the beginning of the Sixties (\cite{Kry60}) have been constructed. The solutions to \eqref{h-o-pde} are sign-varying and their structure has been explored by means of the steepest descent method (\cite{LiWong, AccOrs})  and their representation has been recently given by \cite{ORSDOV}.

For the fractional odd-order stochastic differential equation 
\begin{equation}
\left( \mu + (-1)^n \frac{d^{2n+1}}{d t^{2n+1}} \right)^\beta X(t) = \mathcal{E}(t), \quad n=1,2,\ldots \label{sol4-intro-pde}
\end{equation}
the solution has the structure
\begin{equation}
X(t) = \frac{1}{\Gamma(\beta)} \int_0^\infty dw \, w^{\beta-1} \int_\mathbb{R}dx\,  u_{2n+1}(x,w) \, \mathcal{E}(t+x) \label{sol4-intro}
\end{equation}
where $u_{2n+1}(x,w)$, $x \in \mathbb{R}$, $w\geq 0$ is the fundamental solution to
\begin{align}
\frac{\partial u}{\partial w}(x,w) = (-1)^{n} \frac{\partial^{2n+1} u}{\partial x^{2n+1}}(x,w).
\end{align}
The solutions $u_{2n+1}$ and $u_{2n}$ are substantially different in their behaviour and structure as shown in \cite{ORSDOV} and \cite{Lac03}.

A special attention has been devoted to the case $n=1$ for which \eqref{sol4-intro} takes the interesting form
\begin{equation}
X_3(t) = \frac{1}{\Gamma(\beta)} \int_0^\infty dw \, w^{\beta-1} \int_\mathbb{R}dx\,   \frac{1}{\sqrt[3]{3w}} Ai\left( \frac{x}{\sqrt[3]{3w}} \right) \mathcal{E}(t+x)
\end{equation}
where $Ai(\cdot)$ is the first-type Airy function . The process $X_3$ can also be represented as 
\begin{equation}
X_3(t) = \frac{1}{\mu^\beta}\mathbb{E} \mathcal{E}(t + Y_3(W_\beta)) 
\end{equation}
where the mean value must be meant w.r.t. $Y_3(W_\beta)$ and $Y_3$ is the pseudoprocess related to equation
\begin{equation}
\frac{\partial u}{\partial t} = - \frac{\partial^3 u}{\partial x^3}
\end{equation}
and $W_\beta$ is a Gamma-distributed r.v. with parameters $\beta, \mu$ independent from $Y_3$. The autocovariance function of $X_3$ has the following form
\begin{equation}
\mathbb{E} X_3(t) X_3(t+h) = \frac{\sigma^2}{\mu^{2\beta}}\, \mathbb{E} \left[ \frac{1}{\sqrt[3]{3W_{2\beta}}} Ai\left( \frac{h}{\sqrt[3]{3W_{2\beta}}}\right) \right] \label{cov3ord}
\end{equation}
where $W_{2\beta}$ is the sum of two independent r.v.'s $W_\beta$. 
For the solution to the general odd-order stochastic equation we obtain the covariance function 
\begin{equation}
\mathbb{E} X(t) X(t+h) = \frac{\sigma^2}{\mu^{2\beta}} \, \mathbb{E}\left[ u_{2n+1}(h, W_{2\beta}) \right] \label{cov-aa}
\end{equation}
Of course, the Fourier transform of \eqref{cov-aa} becomes 
\begin{equation}
\label{spect-func-aa}
f(\tau) = \frac{1}{\mu^{2\beta}} \int_{\mathbb{R}} e^{i\tau h} \sigma^2 \, \mathbb{E}\left[ u_{2n+1}(h, W_{2\beta}) \right] dh = \left( \frac{\mu }{\mu + i \tau^{2n+1}} \right)^{2\beta}.
\end{equation}
Stochastic fractional differential equations similar to those dealt with here have been analysed in \cite{Angulo08}, \cite{GayHeyde} and \cite{KLM05}. In our paper we consider equations where different operators are involved.

\section{A stochastic equation involving fractional powers of fractional operators}

In this section we consider the following generalization of the Gay and Heyde equation (see \cite{GayHeyde})
\begin{equation}
	\left(\mu + \frac{d^\alpha}{d (-t)^\alpha} \right)^\beta X(t) \, = \, \mathcal{E}(t), \qquad \beta >0, \, 0<\alpha < 1,\; \mu>0
	\label{1}
\end{equation}
where $\mathcal{E}(t)$, $t>0$, is a Gaussian white noise with
\begin{equation}
	\mathbb{E}\mathcal{E}(t)\mathcal{E}(s) \, = \left\lbrace
	\begin{array}{ll}
	\sigma^2, & t=s\\
	0, & t \neq s
	\end{array} \right. .
\end{equation}
The fractional derivative appearing in \eqref{1} must be meant, for $0 < \alpha \leq 1$, as 
\begin{align}
\frac{d^\alpha}{d (-t)^\alpha} f(t) = & - \frac{1}{\Gamma(1-\alpha)} \frac{d}{dt} \int_t^\infty \frac{f(s)}{(s-t)^\alpha} ds\\
= & \frac{\alpha}{\Gamma(1-\alpha)} \int_0^\infty \frac{f(t) - f(t+w)}{w^{\alpha + 1}} dw.
\end{align}
For information on fractional derivatives of this form, called also Marchaud derivatives, consult  \cite[pag.  111]{SKM}. For $\lambda\geq 0$, we introduce the Laplace transform
\begin{align}
\mathcal{L}\left[\frac{d^\alpha f}{d (-t)^\alpha}\right](\lambda) = \int_0^\infty e^{\lambda t} \frac{d^\alpha}{d (-t)^\alpha} f(t)dt = \lambda^\alpha \mathcal{L}[f](\lambda)
\end{align} 
which can be immediately obtained by considering that
\begin{align}
\mathcal{L}\left[\frac{d^\alpha f}{d (-t)^\alpha}\right](\lambda) =  \frac{\alpha}{\Gamma(1-\alpha)} \int_0^\infty \left( \mathcal{L}[f](\lambda) - e^{-w \lambda} \mathcal{L}[f](\lambda) \right) \frac{dw}{w^{\alpha +1}}
\end{align} 
for a function $f$ such that $e^{\lambda t} f(t) \in L^1([0, \infty))$.

%
%
%
%

%

\begin{theorem}
The representation of a solution to the equation \eqref{1} can be written as
\begin{equation}
	X(t) \, = \, \frac{1}{\Gamma (\beta)} \int_0^\infty dz \int_0^\infty ds \, s^{\beta -1} e^{-s\mu} \, h_\alpha (z, s)\, \mathcal{E}(t+z)
	\label{110THM}
\end{equation}
\end{theorem}
\begin{proof}
The solution to the equation \eqref{1} can be obtained as follows
\begin{align}
	X(t) \, & = \,  \left(\frac{d^\alpha}{d(-t)^\alpha} + \mu \right)^{-\beta} \mathcal{E}(t) \notag \\
	& = \, \frac{1}{\Gamma \left(\beta \right)} \int_0^\infty  s^{\beta -1} e^{-s\mu -s \frac{d^\alpha}{d(-t)^\alpha}} \mathcal{E}(t) \,ds\notag \\
	& = \, \frac{1}{\Gamma \left(\beta \right)} \int_0^\infty  s^{\beta -1} e^{-s\mu} \left\lbrace e^{-s \frac{d^\alpha}{d(-t)^\alpha} } \mathcal{E}(t) \right\rbrace\, ds.
	\label{13}
\end{align}

Now, for the stable subordinator $H^\alpha (t)$, $t>0$, we have that
\begin{align}
	e^{-s \frac{d^\alpha}{d(-t)^\alpha}} \mathcal{E}(t) \, & = \, \mathbb{E}e^{H^\alpha (s) \frac{d}{dt}} \mathcal{E}(t) \notag \\
	& = \, \int_0^\infty dz \, h_\alpha (z, s) \, e^{z \frac{d}{dt}} \mathcal{E}(t) \notag \\
	& = \, \int_0^\infty dz \, h_\alpha (z, s) \, \mathcal{E}(t+z)
	\label{lastStep}
\end{align}
where $h_\alpha (z, s)$ is the probability law of $H^\alpha (s)$, $s>0$. In the last step of \eqref{lastStep} we used the translation property 
\begin{equation}
	e^{z \frac{d}{dt}} \mathcal{E}(t) \, = \, \mathcal{E}(t+z).
	\label{15}
\end{equation}
This is because
\begin{equation}
	e^{z \frac{d}{dt}} \phi(t) \, = \, \sum_{k=0}^\infty \frac{z^k}{k!} \frac{d^k}{dt^k} \, \phi(t).
	\label{}
\end{equation}
In view of the Taylor expansion
\begin{equation}
	f(x) \, = \,  \sum_{k=0}^\infty f^{(k)} (x_0) \frac{(x-x_0)^k}{k!}
	\label{}
\end{equation}
with $x_0 = t$ and $x= t+z$ we have that
\begin{equation}
	e^{z \frac{d}{dt}} \phi(t) \, = \, \phi(t+z)\label{shift}
\end{equation}
which holds for a bounded and continuous function $\phi : [0, \infty) \mapsto [0, \infty)$. Since we can find a sequence of r.v.'s $\{a_j\}_{j \in \mathbb{N}}$ and an orthonormal set, say $\{\phi_j\}_{j \in \mathbb{N}}$, for which \eqref{shift} holds true $\forall\, j$ and such that
\begin{align*}
\lim_{N \to \infty} \mathbb{E}\, \bigg\| \mathcal{E} - \sum_{j=1}^N a_j \phi_j \bigg\|_2=0,
\end{align*}
we can write \eqref{15}. Therefore,
\begin{equation}
	X(t) \, = \, \frac{1}{\Gamma (\beta)} \int_0^\infty dz \int_0^\infty ds \, s^{\beta -1} e^{-s\mu}  \, h_\alpha (z, s)\, \mathcal{E}(t+z)
	\label{110}
\end{equation}
is the formal solution to the fractional equation \eqref{1} with representation, in mean square sense, given by
\begin{equation}
	X(t) \, = \frac{1}{\mu^\beta} \sum_{j \in \mathbb{N}} a_j \, \mathbb{E} [\phi_j(t + H^\alpha(W_\beta))] , \quad t>0.
\end{equation}

\end{proof}

\begin{remark}
For the case $\alpha = 1$, $h_\alpha (z, s) = \delta(z-s)$ where $\delta$ is the Dirac delta function and from \eqref{15} we infer that
\begin{equation}
	X(t) \, = \, \frac{1}{\Gamma \left(\beta \right)} \int_0^\infty e^{-\mu s} s^{\beta -1} \, \mathcal{E}(t+s) \, ds
	\label{}
\end{equation}
solves the fractional equation
\begin{equation}
	\left(\mu - \frac{d}{dt} \right)^\beta X(t) \, = \, \mathcal{E}(t).
	\label{}
\end{equation}
Consult on this point \cite{KLM05}.

A direct proof is also possible because from \eqref{13} we have that
\begin{align}
	X(t) \, & = \, \frac{1}{\Gamma \left(\beta \right)} \int_0^\infty  s^{\beta -1} \, e^{-\mu s} e^{s \frac{d}{dt}} \mathcal{E}(t) \, ds \notag \\
	& = \,\frac{1}{\Gamma \left(\beta \right)} \int_0^\infty  s^{\beta -1} \, e^{-\mu s} \mathcal{E}(t+s) \, ds.
	\label{}
\end{align}
In the last step we applied \eqref{15}.
\end{remark}

\begin{remark}
For $\alpha=1$ and $\beta=1$, we observe that \eqref{110THM} becomes the Ornstein-Uhlenbeck process. 
\end{remark}

Our next step is the evaluation of the Fourier transform of the covariance function of the solution to the differential equation \eqref{1}. Let 
\begin{equation*}
f(\tau) = \int_{-\infty}^{+\infty} e^{i\tau h} Cov_X(h) dh
\end{equation*}
where
\begin{equation*}
Cov_X(h) = \mathbb{E}[X(t+h)X(t)]
\end{equation*}
with $\mathbb{E}X(t)=0$. 

\begin{theorem}
The spectral density of \eqref{110THM} is
\begin{equation}
f(\tau) = \frac{\sigma^2}{\left( \mu^2  + 2 |\tau |^\alpha \mu \cos \frac{\pi \alpha}{2} + |\tau |^{2\alpha} \right)^\beta}, \quad \tau \in \mathbb{R},\; 0 < \alpha < 1,\; \beta>0.
\end{equation}
\end{theorem}
\begin{proof}
The Fourier transform of the covariance function of lag $h=t_2-t_1$ of \eqref{110THM} is given by
\begin{align*}
&  \int_0^\infty \int_0^\infty e^{i \tau \left(t_2 - t_1 \right)} \mathbb{E}X(t_1)X(t_2) \, dt_1 \, dt_2 \, \notag \\
& = \, \frac{1}{\Gamma^2(\beta )} \int_0^\infty \int_0^\infty e^{i \tau \left(t_2 - t_1 \right)} dt_1 \, dt_2 \int_0^\infty dz_1 \int_0^\infty ds_1 \int_0^\infty ds_2 \int_0^\infty dz_2 \, s_1^{\beta -1} \, s_2^{\beta -1} \notag \\
	& \qquad \times e^{- \left(s_1 + s_2 \right)\mu} h_\alpha \left(z_1, s_1 \right)\, h_\alpha \left(z_2, s_2 \right)\, \mathbb{E}\mathcal{E}(t_1 + z_1)\mathcal{E}(t_2 + z_2) \notag 
\end{align*}
where
\begin{align}
\label{CovE}
\mathbb{E}\mathcal{E}(t_1 + z_1)\mathcal{E}(t_2 + z_2)  = \left\lbrace 
\begin{array}{ll}
\sigma^2, & h=z_1-z_2\\
0, & h \neq z_1 - z_2
\end{array} \right. .
\end{align}
Thus, 
\begin{align*}
\int_0^\infty \int_0^\infty e^{i \tau \left(t_2 - t_1 \right)} \mathbb{E}X(t_1)X(t_2) \, dt_1 \, dt_2 = &  \frac{\sigma^2}{\Gamma^2(\beta )} \int_0^\infty dz_1 \int_0^\infty ds_1 \int_0^\infty ds_2 \int_0^\infty dz_2 \, s_1^{\beta -1} \, s_2^{\beta -1} \notag \\
	& \qquad \times e^{- \left(s_1 + s_2 \right)\mu} h_\alpha \left(z_1, s_1 \right)\, h_\alpha \left(z_2, s_2 \right)\, e^{i \tau \left(z_1 - z_2 \right)} \notag .
\end{align*}
By considering the characteristic function of a positively-skewed stable process with law $h_\alpha$, we have that
\begin{equation}
	\int_0^\infty e^{i\tau \, z_1} h_\alpha \left(z_1, s_1 \right)\, dz_1 \, = \, e^{-\left(-i\tau \right)^\alpha s_1} = e^{-s_1 |\tau |^\alpha e^{-i \frac{\pi}{2}\, sgn\, \tau}},
	\label{}
\end{equation}
and
\begin{equation}
	\int_0^\infty e^{-i\tau z_2} h_\alpha \left(z_2, s_2 \right)dz_2 \, = \, e^{-\left(i\tau \right)^\alpha s_2} = e^{-s_2 |\tau |^\alpha e^{i \frac{\pi}{2}\, sgn\, \tau}}.
	\label{}
\end{equation}
Thus, we obtain that
\begin{align*}
&  \int_0^\infty \int_0^\infty e^{i \tau \left(t_2 - t_1 \right)} \mathbb{E}X(t_1)X(t_2) \, dt_1 \, dt_2 \, \notag \\
& = \, \frac{\sigma^2}{\Gamma^2(\beta )}  \int_0^\infty ds_1 \int_0^\infty ds_2\, s_1^{\beta -1} \, s_2^{\beta -1} e^{- \left(s_1 + s_2 \right)\mu} \, e^{-\left(i\tau \right)^\alpha s_2 -\left(-i\tau \right)^\alpha s_1} \notag \\
	& = \, \frac{\sigma^2}{\left( \mu + \left| \tau \right|^\alpha e^{-\frac{i\pi \alpha}{2} \,sgn\,\tau} \right)^\beta \left( \mu + \left| \tau \right|^\alpha e^{\frac{i\pi \alpha}{2}\,sgn\,\tau} \right)^\beta} \notag \\
	& = \, \frac{\sigma^2}{\left( \mu^2  + 2 |\tau |^\alpha \mu \cos \frac{\pi \alpha}{2} + |\tau |^{2\alpha} \right)^\beta}.
\end{align*}
\end{proof}

\begin{remark}
In the special case $\alpha = 1$ the result above simplifies and yields
\begin{align}
f(\tau) = & \frac{\sigma^2}{ (\mu^2 + \tau^2 )^\beta }.\label{222becomes}
\end{align}
We note that for $\beta=1$, \eqref{222becomes} becomes the spectral function of the Ornstein-Uhlenbeck process. Processes with the spectral function $f$ are dealt with, for example, in \cite{Angulo08} where also space-time random fields governed by stochastic equations are considered. The covariance function is given by
\begin{align*}
Cov_X(h) = & \frac{1}{2\pi} \int_\mathbb{R} e^{-i \tau h} f(\tau) d\tau \\
= & \frac{\sigma^2}{2\pi} \int_\mathbb{R} e^{-i \tau h} \left( \frac{1}{\Gamma(\beta)}\int_0^\infty z^{\beta - 1} e^{-z \mu^2 - z\tau^2} dz \right) d\tau\\
= &  \frac{\sigma^2}{\Gamma(\beta)} \int_0^\infty z^{\beta - 1} e^{-z \mu^2} \left( \frac{1}{2\pi} \int_\mathbb{R} e^{-i \tau h - z \tau^2} d\tau \right) dz\\
= & \frac{\sigma^2}{\Gamma(\beta)} \int_0^\infty z^{\beta - 1} e^{-z \mu^2} \frac{e^{-\frac{h^2}{4z}}}{\sqrt{4\pi z}} dz\\
= &  \frac{\sigma^2}{2\Gamma(\beta) \Gamma(\frac{1}{2})} \int_0^\infty z^{\beta- \frac{1}{2}-1} e^{-z \mu^2 - \frac{h^2}{4z}}dz\\
= & \frac{\sigma^2}{\Gamma(\beta) \Gamma(\frac{1}{2})} \left( \frac{|h|}{2 \mu}\right)^{\beta - \frac{1}{2}} K_{\beta - \frac{1}{2}}\left( \mu |h| \right), \quad h\geq 0
\end{align*}
where $K_\nu$ is the modified Bessel function with intergal representation given by
\begin{equation}
\int_0^\infty x^{\nu -1} \exp \left\lbrace -\beta x^p - \alpha x^{-p} \right\rbrace dx = \frac{2}{p} \left( \frac{\alpha}{\beta} \right)^\frac{\nu}{2p} K_\frac{\nu}{p} \left(2 \sqrt{\alpha \beta} \right), \quad p,\alpha,\beta,\nu >0 \label{formula:K}
\end{equation}
(see for example \cite{GR}, formula 3.478). We observe that $K_\nu=K_{-\nu}$ and $K_\frac{1}{2}(x) = \sqrt{\frac{\pi}{2x}}e^{-x}$. Moreover,
\begin{equation}
K_\nu(x) \approx \frac{2^{\nu - 1} \Gamma(\nu)}{x^{\nu}} \quad \textrm{ for } \quad x \to 0^+
\end{equation}
(\cite[pag. 136]{leb}) and
\begin{equation}
K_\nu(x) \approx \sqrt{\frac{\pi}{2x}} e^{-x} \quad  \textrm{ for } \quad x \to \infty.
\end{equation}
Thus, we get that
\begin{equation}
Cov_X(h) \approx \mu^{1 - 2\beta}, \quad  \textrm{ for } \quad h \to 0^+
\end{equation}
and
\begin{equation}
Cov_X(h) \approx \left( \frac{h}{ \mu}\right)^{\beta } \frac{1}{h} e^{-\mu h}, \quad \textrm{ for } \quad h \to \infty.
\end{equation}
\end{remark}

We study the covariance of \eqref{sol1-intro}. Recall that, a stable process $S$ of order $\alpha$ with density $g$ is characterized by
\begin{align*}
\widehat{g}(\xi, t) = \mathbb{E} e^{i\xi S(t)} = e^{- \sigma^2 |\xi |^\alpha t}, \quad \alpha \in (0,2].
\end{align*}
Consider two independent stable processes $S_1(w)$, $S_2(w)$, $w\geq 0$, with $\sigma^2_1=1$ and $\sigma^2_2=2\mu \cos \frac{\pi \alpha}{2}$. Let $g_1(x,w)$, $x \in \mathbb{R}$, $w\geq0$ and $g_2(x,w)$, $x \in \mathbb{R}$, $w\geq0$ be the corresponding density laws. Then, the following result holds true.
\begin{theorem}
The covariance function of \eqref{sol1-intro} is 
\begin{equation}
Cov_X(h) = \frac{\sigma^2}{\Gamma(\beta)} \int_0^\infty w^{\beta -1} e^{-w\mu^2 } \int_{-\infty}^{+\infty} g_1(h-z, w) g_2(z, w) dz\, dw
\end{equation}
or
\begin{equation}
Cov_X(h) = \frac{\sigma^2}{\mu^{2\beta}} \mathbb{E} g_{S_1+S_2}(h, W_\beta)
\end{equation}
and $W_\beta$ is a gamma r.v. with parameters $\mu^2, \beta$.
\end{theorem}
\begin{proof}
Notice that
\begin{align*}
f(\tau) = \frac{\sigma^2}{\Gamma(\beta)} \int_0^\infty w^{\beta -1} e^{-w (\mu^2 + 2|\tau |^\alpha \mu \cos \frac{\pi \alpha}{2} + |\tau |^{2\alpha})} dw
\end{align*}
where
\begin{align*}
e^{- 2\mu \cos \frac{\pi \alpha}{2} |\tau |^\alpha w }= \mathbb{E} e^{i\tau S_2(w)}= \widehat{g_2}(\tau, w) \quad \textrm{and} \quad  e^{-|\tau |^{2\alpha} w} = \mathbb{E} e^{i\tau S_1(w)}=\widehat{g_1}(\tau, w).
\end{align*}
Thus, 
\begin{align*}
f(\tau) = \frac{\sigma^2}{\mu^{2\beta}} \mathbb{E} [\widehat{g_1}(\tau, W_\beta)\,\widehat{g_2}(\tau, W_\beta)]
\end{align*}
from which, we immediately get that
\begin{align*}
CovX(h) = & \frac{\sigma^2}{\mu^{2\beta}}  \mathbb{E}\left[ \int_{-\infty}^{+\infty} g_1(h-z, W_\beta) g_2(z, W_\beta) dz \right]\\
= & \frac{\sigma^2}{\Gamma(\beta)} \int_0^\infty w^{\beta -1} e^{-w\mu } \int_{-\infty}^{+\infty} g_1(h-z, w) g_2(z, w) dz\, dw
\end{align*}
\end{proof}

\section{Fractional powers of higher-order operators}

We focus our attention on the following equation
\begin{equation}
\left( \mu - \frac{d^2}{d t^2} \right)^{\beta} X(t) = \mathcal{E}(t), \quad t\geq 0,\; \mu>0,\; \beta>0  \label{eqn3-attention}
\end{equation}
that is, to the equation \eqref{sol2-intro-pde} for $n=1$.
\begin{theorem}
\label{aaawwqqT}
The representation of a solution to the equation \eqref{eqn3-attention} can be written as
\begin{equation}
	X(t) \, = \, \frac{1}{\Gamma \left(\beta \right)} \int_0^\infty w^{\beta -1} e^{-\mu w} \int_{-\infty}^{+\infty} u_{2}(x,w)  \mathcal{E}(t+x) \, dx\, dw, \quad \beta>0,\; \mu>0.
	\label{aaawwqq}
\end{equation}
Moreover, the spectral function of \eqref{eqn3-attention} reads
\begin{equation}
f(\tau) = \frac{\sigma^2}{(\mu + \tau^2)^{2\beta}}
\end{equation}
and the corresponding covariance function has the form
\begin{align}
Cov_X(h) = & \frac{\sigma^2}{\mu^{2\beta}} \mathbb{E} \left[ \frac{e^{-\frac{h^2}{4W_{2\beta}}}}{2\sqrt{\pi W_{2\beta}}} \right] =  \frac{2\sigma^2}{\Gamma\left(2\beta \right)} \left( \frac{|h|}{2 \sqrt{\mu}} \right)^{2\beta} K_{2\beta} (|h| \sqrt{\mu})
\end{align}
where $W_{2\beta}$ is a gamma r.v. with parameters $\mu$, $2\beta$.
\end{theorem}
\begin{proof}
We can formally write
\begin{equation}
	e^{w \frac{d^2}{dt^2}} \, = \, \int_{-\infty}^\infty e^{x\frac{d}{dt}} \frac{e^{-\frac{x^2}{4w}}}{2\sqrt{\pi w}} \, dx 
	\label{}
\end{equation}
so that from \eqref{eqn3-attention} we have that
\begin{align}
	X(t) \, & = \, \frac{1}{\Gamma \left(\beta \right)} \int_0^\infty e^{-\mu w} w^{\beta -1} \, dw \int_{-\infty}^\infty \frac{e^{-\frac{x^2}{4w}}}{2\sqrt{\pi w}} \, e^{x \frac{d}{dt}} \mathcal{E}(t) \, dx \notag \\
	& = \, \frac{1}{\Gamma \left(\beta \right)} \int_0^\infty e^{-\mu w} w^{\beta -1} dw \int_{-\infty}^\infty \frac{e^{-\frac{x^2}{4w}}}{2\sqrt{\pi w}} \mathcal{E}(t+x) \, dx.
	\label{}
\end{align}
By observing that
\begin{align*}
\mathbb{E}\mathcal{E}(t+x_1) \mathcal{E}(t+ h + x_2) = \left\lbrace
\begin{array}{ll}
\sigma^2, & x_1 -x_2 = h\\
0, & otherwise
\end{array}
\right.
\end{align*}
we can write
\begin{align*}
\mathbb{E} X(t)X(t+h) = & \frac{\sigma^2}{\Gamma^2 \left(\beta \right)} \int_0^\infty e^{-\mu w_1} w_1^{\beta -1} dw_1 \int_0^\infty e^{-\mu w_2} w_2^{\beta -1} dw_2   \int_{-\infty}^\infty  \frac{e^{-\frac{x_1^2}{4w_1}}}{2\sqrt{\pi w_1}} \frac{e^{-\frac{(h - x_1)^2}{4w_2}}}{2\sqrt{\pi w_2}} dx_1\\
= & \frac{\sigma^2}{\Gamma^2 \left(\beta \right)} \int_0^\infty e^{-\mu w_1} w_1^{\beta -1} dw_1 \int_0^\infty e^{-\mu w_2} w_2^{\beta -1} dw_2  \frac{e^{-\frac{h^2}{4(w_1+w_2)}}}{2\sqrt{\pi (w_1+w_2)}}\\
= &  \frac{\sigma^2}{\mu^{2\beta}} \mathbb{E} \left[ \frac{e^{-\frac{h^2}{4(W_1+W_2)}}}{2\sqrt{\pi (W_1+W_2)}} \right] \\
= & \frac{\sigma^2}{\mu^{2\beta}} \mathbb{E} \left[ \frac{e^{-\frac{h^2}{4W}}}{2\sqrt{\pi W}} \right] \\
= & \frac{\sigma^2}{\Gamma\left(2\beta \right)} \int_0^\infty \frac{e^{-\frac{h^2}{4w}}}{2\sqrt{\pi w}} w^{2\beta -1} e^{-\mu w} dw\\
= & \frac{2\sigma^2}{\Gamma\left(2\beta \right)} \left( \frac{h}{2 \sqrt{\mu}} \right)^{2\beta} K_{2\beta} (h \sqrt{\mu})
\end{align*}
We notice that
\begin{align*}
Cov_X(h) =\frac{\sigma^2}{\mu^{2\beta}} P(B(W_{2\beta}) \in dh)/dh
\end{align*}
where $B(W_{2\beta})$ is a Brownian motion with random time $W_{2\beta}$. Thus, we obtain that
\begin{align*}
f(\tau) =& \int_{-\infty}^{\infty} e^{i\tau h} Cov_X(h)\, dh = \frac{\sigma^2}{\Gamma\left(2\beta \right)} \int_0^\infty e^{-w \tau^2} w^{2\beta -1} e^{-\mu w} dw = \frac{\sigma^2}{(\mu + \tau^2)^{2\beta}}.
\end{align*}
\end{proof}

An alternative representation of the covariance function above reads
\begin{align*}
\mathbb{E} X(t)X(t+h) = & \frac{\sigma^2}{\Gamma^2 \left(\beta \right)} \int_0^\infty e^{-\mu w_1} w_1^{\beta -1} dw_1 \int_0^\infty e^{-\mu w_2} w_2^{\beta -1} dw_2   \int_{-\infty}^\infty  \frac{e^{-\frac{x_1^2}{4w_1}}}{2\sqrt{\pi w_1}} \frac{e^{-\frac{(x_1-h)^2}{4w_2}}}{2\sqrt{\pi w_2}} dx_1\\
= & 4\sigma^2 \, \int_{-\infty}^{+\infty} \left( \frac{|x_1||x_2 -h|}{4 \mu} \right)^{\beta - \frac{1}{2}}\, K_{\beta - \frac{1}{2}}(\sqrt{\mu}|x_1|) \, K_{\beta - \frac{1}{2}}(\sqrt{\mu} |x_1 - h|) \, dx_1.
\end{align*}

We now pass to the general even-order fractional equation \eqref{sol2-intro-pde}.

\begin{theorem}
\label{aaawwwqqqT}
The representation of a solution to the equation \eqref{sol2-intro-pde} can be written as
\begin{equation}
	X(t) \, = \, \frac{1}{\Gamma \left(\beta \right)} \int_0^\infty w^{\beta -1} e^{-\mu w } \int_{-\infty}^{+\infty}  u_{2n}(x, w) \mathcal{E}(t+x)\, dx\, dw, \quad \beta >0,\; \mu>0.
	\label{aaawwwqqq}
\end{equation}
Moreover, the spectral function of \eqref{aaawwwqqq} is
\begin{equation}
f(\tau) = \frac{\sigma^2}{(\mu + \tau^{2n})^{2\beta}}  
\end{equation}
and the covariance function reads
\begin{equation}
Cov_X(h) =  \frac{\sigma^2}{\mu^{2\beta}} \mathbb{E} \left[ u_{2n}(h, W_{2\beta}) \right]
\end{equation}
where $W_{2\beta}$ is a gamma r.v. with parameters $\mu, 2\beta$.
\end{theorem}
\begin{proof}
The solution $u_{2n} (x, t)$ to	
\begin{equation}
	\frac{\partial}{\partial t} u_{2n} \, = \, \left(-1 \right)^{n+1} \frac{\partial^{2n}}{\partial x^{2n}} u_{2n}
	\label{}
\end{equation}
has Fourier transform
\begin{equation}
	U(\beta, t) \, = \, e^{(-1)^{n+1} \left(-i \beta \right)^{2n} t} \, = \, e^{-\beta^{2n} t}.
	\label{U2n}
\end{equation}
We write
\begin{equation}
	e^{-w \frac{\partial^{2n}}{\partial t^{2n}}} \, = \, \int_{-\infty}^\infty e^{ix \frac{\partial}{\partial t}} u_{2n} (x, w) \, dx.
	\label{}
\end{equation}
Since
\begin{equation}
	U(-i\beta, t) \, = \, e^{- (-1)^n \beta^{2n} t},
	\label{}
\end{equation}
we also write
\begin{equation}
	e^{-w (-1)^n \frac{\partial^{2n}}{\partial t^{2n}}} \, = \, \int_{-\infty}^\infty e^{x \frac{\partial}{\partial t}} u_{2n} (x, w) \, dx.
	\label{}
\end{equation}
In conclusion, we have that
\begin{align}
X(t)  & = \left( \mu + (-1)^n \frac{\partial^{2n}}{\partial t^{2n}} \right)^{-\beta} \mathcal{E}(t)\\  
& = \frac{1}{\Gamma (\beta )} \int_0^\infty dw \, e^{-\mu w} w^{\beta -1} \left( \int_{-\infty}^{+\infty} dx\, u_{2n} (x, w)\, e^{x \frac{\partial}{\partial t}}  \mathcal{E}(t) \right) \notag \\
& = \, \frac{1}{\Gamma (\beta )}\int_0^\infty dw \, e^{-\mu w} w^{\beta -1} \int_{-\infty}^{+\infty} dx\, u_{2n}\left(x, w \right) \, \mathcal{E}(t+x)
\end{align}
and this confirms \eqref{aaawwwqqq}. 

From \eqref{aaawwwqqq}, in view of \eqref{CovE}, we obtain
\begin{align*}
	\mathbb{E} X(t)X(t+h) = &  \frac{\sigma^2}{\Gamma^2 (\beta )}\int_0^\infty dw_1\, w_1^{\beta -1} e^{-\mu w_1} \int_0^\infty dw_2 \, w_2^{\beta -1} e^{-\mu w_2}  \\
	 \cdot & \int_{-\infty}^{+\infty} dx_1\, u_{2n}\left(x_1, w_1 \right)  \int_{-\infty}^{+\infty} dx_2\, u_{2n}\left(x_2, w_2 \right)  \, \delta(x_2-x_1+ h)\\
	 = &  \frac{\sigma^2}{\Gamma^2 (\beta )}\int_0^\infty dw_1\,w_1^{\beta -1} e^{-\mu w_1} \int_0^\infty dw_2 \,  w_2^{\beta -1} e^{-\mu w_2}  \\
	 \cdot & \int_{-\infty}^{+\infty} dx_1\, u_{2n}\left(x_1, w_1 \right) \, u_{2n}\left(x_1-h, w_2 \right)  \\
	 = &  \frac{\sigma^2}{\Gamma^2 (\beta )}\int_0^\infty dw_1\,w_1^{\beta -1} e^{-\mu w_1} \int_0^\infty dw_2 \,  w_2^{\beta -1} e^{-\mu w_2} \, u_{2n}(h, w_1+w_2)\\
	 = & \frac{\sigma^2}{\mu^{2\beta}} \mathbb{E} u_{2n}(h, W_1+W_2).
\end{align*}
By following the same arguments as in the previous proof, we get that
\begin{align*}
\mathbb{E} X(t)X(t+h) = & \frac{\sigma^2}{\mu^{2\beta}} \mathbb{E} u_{2n}(h, W_{2\beta}) = \frac{\sigma^2}{\Gamma (2\beta )}\int_0^\infty dw\,w^{2\beta -1} e^{-\mu w} \, u_{2n}(h, w)
\end{align*}
The spectral density of $X(t)$ is therefore
\begin{align*}
f(\tau) = &  \frac{\sigma^2}{\Gamma (2\beta )}\int_0^\infty dw\,w^{2\beta -1} e^{-\mu w - \tau^{2n} w} = \frac{\sigma^2}{(\mu + \tau^{2n})^{2\beta}} .
\end{align*}
\end{proof}

Theorem \ref{aaawwwqqqT} extends the results of Theorem \ref{aaawwqqT} when even-order heat-type equations are involved.

We now pass to the study of the equation \eqref{sol4-intro-pde} for $n=1$ and $\kappa = \mp 1$, 
\begin{equation}
\left( \mu + \kappa \frac{d^3}{d t^3} \right)^\beta X(t) = \mathcal{E}(t). \label{eq-stella}
\end{equation}
\begin{theorem}
The representation of a solution to the equation \eqref{eq-stella} can be written as
\begin{equation}
	X(t) \, = \, \frac{1}{\Gamma \left(\beta \right)} \int_0^\infty e^{-\mu w} dw \, w^{\beta -1} \int_{-\infty}^\infty \frac{1}{\sqrt[3]{3w}} \textrm{Ai} \left(\frac{\kappa x}{\sqrt[3]{3w}}  \right)\, \mathcal{E}(t+x) \, dx \, dw, \quad \beta>0,\; \mu>0.
	\label{118}
\end{equation}
Moreover,
\begin{equation}
f(\tau) = \frac{\sigma^2}{(\mu +\kappa i \tau^3)^{2\beta}}
\end{equation}
and
\begin{equation}
Cov_{X}(h) = \frac{\sigma^2}{\mu^{2\beta}} \mathbb{E} \left[ \frac{\sigma^2}{\sqrt[3]{3W}} \textrm{Ai} \left(\frac{-\kappa h}{\sqrt[3]{3W}}  \right)\right]
\end{equation}
where $\textrm{Ai}(x)$ is the Airy function and $W$ is a gamma-distributed r.v. with parameters $2\beta$ and $\mu$.
\end{theorem}
\begin{proof}
By following the approach adopted above, after some calculation we can write that
\begin{equation}
	X^{-}(t) \, = \, \frac{1}{\Gamma \left(\beta \right)} \int_0^\infty w^{\beta -1} e^{-\mu w + w \frac{d^3}{dt^3}} \, \mathcal{E}(t) \, dw
	\label{118a}
\end{equation}
is the solution to
\begin{equation}
\left( \mu - \frac{d^3}{d t^3} \right)^\beta X(t) = \mathcal{E}(t)
\end{equation}
whereas
\begin{equation}
	X^{+}(t) \, = \, \frac{1}{\Gamma \left(\beta \right)} \int_0^\infty  w^{\beta -1} e^{-\mu w - w \frac{d^3}{dt^3}} \, \mathcal{E}(t) \, dw
	\label{118b}
\end{equation}
is the solution to
\begin{equation}
\left( \mu + \frac{d^3}{d t^3} \right)^\beta X(t) = \mathcal{E}(t)
\end{equation}
The third-order heat type equation
\begin{equation}
	\frac{\partial}{\partial t} u \, = \, \kappa\, \frac{\partial^3}{\partial x^3} u, \qquad u(x, 0) \, = \, 0,
	\label{3-ord-eq}
\end{equation}
has solution, for $\kappa=-1$,
\begin{equation}
	u(x, t) \, = \, \frac{1}{\sqrt[3]{3t}} \, \textrm{Ai} \left(\frac{x}{\sqrt[3]{3t}} \right), \qquad x \in \mathbb{R}, t>0,
	\label{}
\end{equation}
with Fourier transform
\begin{equation}
	\int_{-\infty}^\infty e^{i\beta x} u(x, t) \, dx \, = \, e^{-it \beta^3}.
	\label{fur-3ord}
\end{equation}
Formula \eqref{fur-3ord} leads to the integral
\begin{align*}
	\int_{-\infty}^\infty e^{\theta x} u(x, t) \, dx \, = \, e^{t \theta^3}, \quad \theta \in \mathbb{R}
\end{align*}
because of the asymptotic behaviour of the Airy function (see \cite{AccOrs} and \cite{LiWong}).
The solution to \eqref{sol4-intro-pde} with $n=1$ (that is $\kappa=-1$) is therefore \eqref{118a}. 

The equation \eqref{3-ord-eq} has solution, for $\kappa=+1$, given by
\begin{equation}
	u(x, t) \, = \, \frac{1}{\sqrt[3]{3t}} \, \textrm{Ai} \left(\frac{-x}{\sqrt[3]{3t}} \right), \qquad x \in \mathbb{R}, t>0.
	\label{}
\end{equation}
Thus, by following the same reasoning as before, we arrive at
\begin{align*}
	\int_{-\infty}^\infty e^{\theta x} u(x, t) \, dx \, = \, e^{-t \theta^3}, \quad \theta \in \mathbb{R}
\end{align*}
and we obtain that \eqref{118b} solves \eqref{eq-stella} with $\kappa=+1$ is \eqref{118b}.

In light of \eqref{CovE} we get
\begin{align*}
\mathbb{E} [X^-(t) \,X^-(t+h)] \, = & \frac{\sigma^2 }{\Gamma^2 \left(\beta \right)} \int_0^\infty e^{-\mu w_1} dw_1 \, w_1^{\beta -1} \int_0^\infty e^{-\mu w_2} dw_2 \, w_2^{\beta -1} \\
\cdot & \int_{-\infty}^\infty \frac{1}{\sqrt[3]{3w_1}} \textrm{Ai} \left(\frac{x_1}{\sqrt[3]{3w_1}}  \right)  \frac{1}{\sqrt[3]{3w_2}} \textrm{Ai} \left(\frac{x_1- h}{\sqrt[3]{3w_2}}  \right)\, dx_1\\
= &  \frac{\sigma^2}{\Gamma^2 \left(\beta \right)} \int_0^\infty e^{-\mu w_1} dw_1 \, w_1^{\beta -1} \int_0^\infty e^{-\mu w_2} dw_2 \, w_2^{\beta -1}\\
\cdot & \frac{1}{\sqrt[3]{3(w_1 + w_2)}} \textrm{Ai} \left(\frac{h}{\sqrt[3]{3(w_1+w_2)}}  \right)\\
= & \frac{\sigma^2}{\mu^{2\beta}} \mathbb{E} \left[ \frac{1}{\sqrt[3]{3W_{2\beta}}} \textrm{Ai} \left(\frac{h}{\sqrt[3]{3W_{2\beta}}}  \right) \right].
\end{align*}
From the Fourier transform \eqref{fur-3ord}, we get that
\begin{align*}
f^{-}(\tau) = & \frac{\sigma^2}{\mu^{2\beta}} \int_{\mathbb{R}} e^{i \tau h}\,  \mathbb{E} \left[ \frac{1}{\sqrt[3]{3W_{2\beta}}} \textrm{Ai} \left(\frac{h}{\sqrt[3]{3W_{2\beta}}}  \right) \right] dh \\
= & \frac{\sigma^2}{\mu^{2\beta}} \mathbb{E} \left[ e^{-i \tau^3 W_{2\beta}} \right] \\
= & \frac{\sigma^2}{(\mu + i \tau^3)^{2\beta}} \\
= & \frac{\sigma^2e^{-i2\beta\arctan \frac{\tau^3}{\mu}}}{(\mu^2 + \tau^6)^\beta}.
\end{align*}
Also, we obtain that
\begin{align*}
\mathbb{E} [X^+(t) \,X^+(t+h)] \, = & \frac{\sigma^2}{\mu^{2\beta}}\mathbb{E} \left[ \frac{1}{\sqrt[3]{3W_{2\beta}}} \textrm{Ai} \left(\frac{-h}{\sqrt[3]{3W_{2\beta}}}  \right) \right].
\end{align*}
with 
\begin{equation}
f^{+}(\tau)=\frac{\sigma^2}{(\mu - i \tau^3)^{2\beta}} = \frac{\sigma^2e^{+i2 \beta\arctan \frac{\tau^3}{\mu}}}{(\mu^2 + \tau^6)^\beta}.
\end{equation}
\end{proof}

\begin{theorem}
The representation of a solution to the following equation 
\begin{align*}
\left( \mu + \kappa \frac{\partial^{2n+1}}{\partial t^{2n+1}} \right)^\beta X(t) = \mathcal{E}(t), \quad n=1,2, \ldots
\end{align*}
reads
\begin{align*}
X(t) = \frac{1}{\Gamma(\beta)} \int_0^\infty w^{\beta -1} e^{-\mu w} \int_{-\infty}^{+\infty} u_{2n+1}(\kappa x, w)  \mathcal{E}(t+x) dw dx, \quad \beta>0,\; \mu>0.
\end{align*}
Moreover, the covariance function 
\begin{align*}
Cov_X(h) = \frac{\sigma^2}{\mu^{2\beta}} \mathbb{E} u_{2n+1}(\kappa h, W_{2\beta}).
\end{align*}
has Fourier transform
\begin{align*}
f(\tau) = \frac{\sigma^2}{(\mu + \kappa i \tau^{2n+1})^{2\beta}} = \frac{\sigma^2 e^{-i2\beta \kappa \arctan \frac{\tau^{2n+1}}{\mu}}}{(\mu^2 + \tau^{2(2n+1)})^\beta}. 
\end{align*}  

\end{theorem}
\begin{proof}
The proof follows the same lines as in the previous theorem.
\end{proof}

\begin{figure}
\includegraphics[scale=.8]{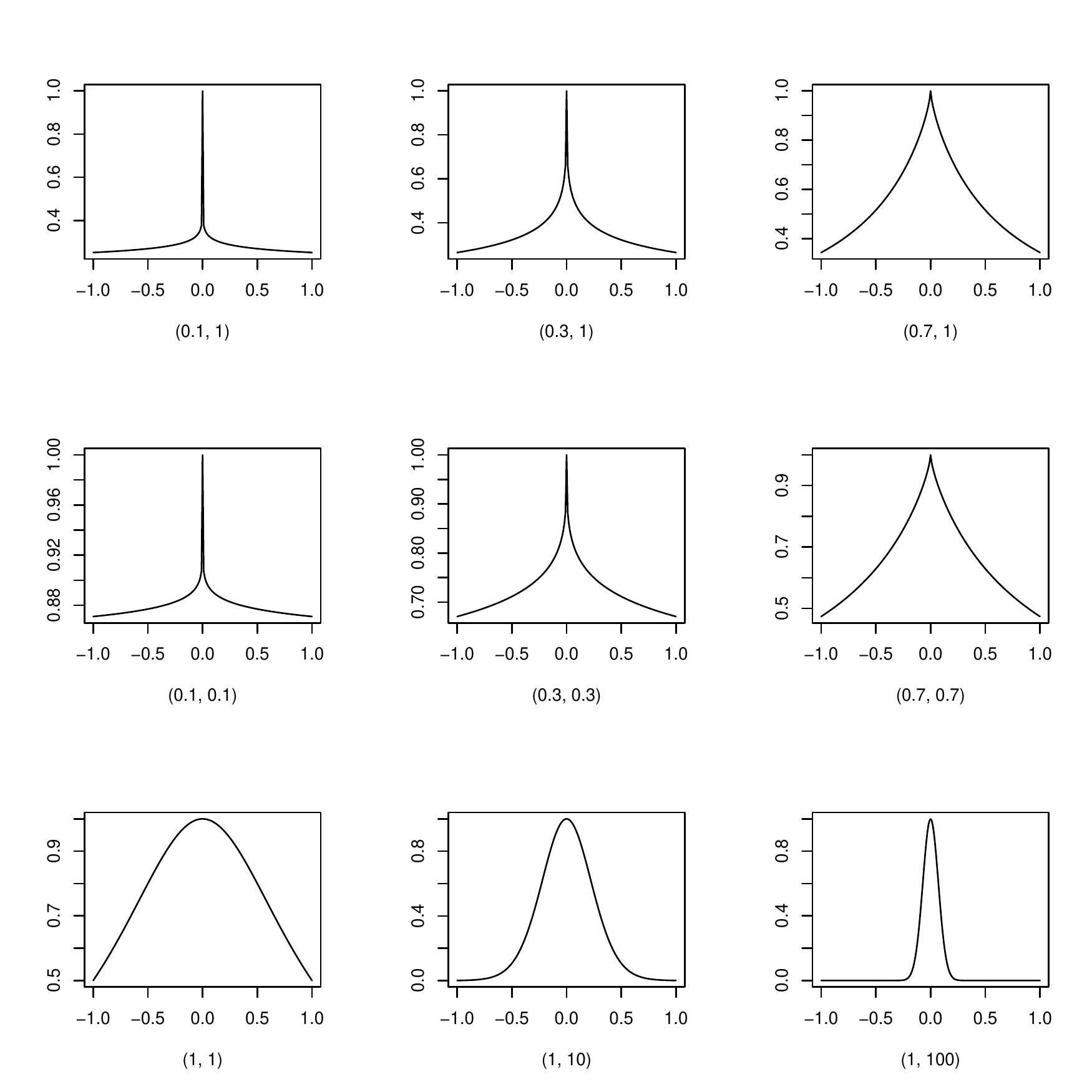}
\caption{The spectral function \eqref{spectralfFig} with different values for the parameters $(\alpha, \beta)$.}
\end{figure}

\end{document}